\documentclass{amsart}
\usepackage{amsmath}
\usepackage{amsfonts}
\usepackage{amssymb}
\usepackage{mathtools}
\usepackage{enumitem}
\usepackage{hyperref}
\usepackage[all]{xy}

\usepackage{prettyref}

\title{On the arithmetic of a family of degree-two  K3 surfaces}

\author[F. Bouyer]{Florian Bouyer}
\address[Florian Bouyer]{Floor 4 \\
Howard House \\
Queen's Avenue \\
Bristol \\
BS8 1SD \\
UK}
\email{f.j.s.c.bouyer@gmail.com}

\author[E. Costa]{Edgar Costa}
\address[Edgar Costa]{Department of Mathematics \\
Dartmouth College \\
27 N Main Street \\
6188 Kemeny Hall\\
Hanover, NH 03755-3551\\
U.S.A.}
\email{edgarcosta@math.dartmouth.edu}

\author[D. Festi]{Dino Festi}
\address[Dino Festi]{Institut f\"{u}r Mathematik \\ Johannes Gutenberg--Universit\"{a}t \\
Staudingerweg 9,
55128 Mainz,
Germany }
\email{dinofesti@gmail.com}

\author[C. Nicholls]{Christopher Nicholls}
\address[Christopher Nicholls]{
Mathematical Institute\\
University of Oxford\\
Andrew Wiles Building\\
Radcliffe Observatory Quarter\\
Woodstock Road\\
Oxford\\
OX2 6GG \\
UK
}
\email{christopher.nicholls@balliol.ox.ac.uk}

\author[M. West]{Mckenzie West}
\address[Mckenzie West]{
	Department of Mathematics\\
	Kalamaozo College\\
	1200 Academy St\\
	Kalamazoo, MI 49006\\
	U.S.A.
}
\email{mckenzie.west@kzoo.edu}

\newcommand{\QQ}{\mathbb{Q}}
\newcommand{\Q}{\mathbb{Q}}

\newcommand{\RR}{\mathbb{R}}
\newcommand{\ZZ}{\mathbb{Z}}
\renewcommand{\AA}{\mathbb{A}}
\newcommand{\Z}{\mathbb{Z}}

\newcommand{\FF}{\mathbb{F}}

\newcommand{\calX}{\mathcal{X}}
\newcommand{\genericFiber}{\calX}
\newcommand{\genericFiberbar}{{\overline{\genericFiber}}}
\newcommand{\ellipticfibration}{\mathcal{Y}}

\newcommand{\pp}{\mathfrak p}
\newcommand{\branchCurve}{\mathcal{B}}

\newcommand{\PP}{\mathbb{P}}
\newcommand{\kbar}{{\overline{k}}}

\newcommand{\Xetabar}{{\overline{\genericFiber}}}

\newcommand{\Kbar}{\overline{K}}
\newcommand{\Qbar}{{\overline{\QQ}}}

\newcommand{\calO}{\mathcal{O}}

\DeclareMathOperator{\Kum}{Kum}

\DeclareMathOperator{\rank}{rk}
\DeclareMathOperator{\Pic}{Pic}

\DeclareMathOperator{\End}{End}
\let\Hom\relax %\Hom is defined in psp class, redefined for consistency
\DeclareMathOperator{\Hom}{Hom} 
\DeclareMathOperator{\Aut}{Aut}
\DeclareMathOperator{\Gal}{Gal}
\DeclareMathOperator{\Br}{Br}
\DeclareMathOperator{\disc}{disc}
\newcommand{\coho}[1]{{H}^{#1}}
\DeclareMathOperator{\OO}{O}

\usepackage{mathrsfs} % For sheaves
\newcommand{\xyR}[1]{
 \xydef@\xymatrixrowsep@{#1}}
\newcommand{\xyC}[1]{
  \xydef@\xymatrixcolsep@{#1}}
\newcommand{\XBeq}[1]{x^6+y^6+z^6+ #1 x^2y^2z^2}

\newtheorem{theorem}{Theorem}[section]
\newtheorem{corollary}[theorem]{Corollary}
\newtheorem{proposition}[theorem]{Proposition}
\newtheorem{lemma}[theorem]{Lemma}
\newtheorem{computation}[theorem]{Computation}

\newtheorem{remark}[theorem]{Remark}

\begin{document}
\begin{abstract}
Let $\mathbb{P}$ denote the weighted projective space with weights $(1,1,1,3)$ over the rationals, with coordinates $x,y,z,$ and $w$; let $\mathcal{X}$ be the generic element of the family of surfaces in $\mathbb{P}$ given by
\begin{equation*}
  X\colon w^2=x^6+y^6+z^6+tx^2y^2z^2.
\end{equation*}
The surface $\mathcal{X}$ is a K3 surface over the function field $\mathbb{Q}(t)$. In this paper, we explicitly compute the geometric Picard lattice of $\mathcal{X}$,  together with its Galois module structure, as well as derive more results on the arithmetic of $\mathcal{X}$ and other elements of the family $X$.
\end{abstract}
\maketitle

\section{Introduction}
\label{section:introduction}
K3 surfaces are sometimes called ``surfaces of intermediate type'',
as they are neither birational to $\PP^2$ nor of general type;
that is,
they lie between those surfaces whose arithmetic and geometry is well
understood and  those surfaces whose arithmetic and geometry is still largely obscure,
occupying a position similar to elliptic curves among curves.
Even though interest in K3 surfaces has recently increased, 
very basic questions about their arithmetic are still unanswered.
For example, it is not known if there are K3 surfaces with finitely many
rational points, or if there are K3 surfaces with rational points that are not
Zariski dense.

An important tool in understanding the arithmetic and the geometry of a K3 surface is its Picard lattice.
The Picard lattice encodes information about the existence of elliptic fibrations, the potential density of rational points and, if the surface is defined over a global field, also the existence of a Brauer--Manin obstruction to the Hasse principle on the surface.

Let $\PP=\PP_\QQ (1,1,1,3)$ denote the weighted projective space over $\QQ$ with weights $(1,1,1,3)$ and coordinates $x,y,z,$ and $w$;
let $\AA^1$ denote the affine line over $\QQ$, with coordinate $t$. Consider the following family of degree-two K3 surfaces
$$
X\colon w^2=x^6+y^6+z^6+tx^2y^2z^2 \subset \PP \times \AA^1.
$$ 
Let $\genericFiber$ be the generic element of the family, and let $\genericFiberbar$ denote its base change to $\overline{\QQ(t)}$.
Then $\genericFiber$  is a K3 surface over the field $\QQ (t)$ (see Proposition~\ref{prop:Xfibers}).
In this paper we explicitly compute the geometric Picard lattice of $\genericFiber$, together with its Galois module structure.

\begin{theorem}\label{theorem:main}
    The geometric Picard lattice $\Pic \genericFiberbar$ is isometric to the unique
    (up to isometries) lattice with rank $19$, signature $(1,18)$, determinant $2^5\cdot 3^3$, and discriminant group isomorphic to $\ZZ/6 \ZZ \times \left(\ZZ/12\ZZ\right)^2$.
\end{theorem}

%FIXME
We prove Theorem~\ref{theorem:main} as follows. In Section~\ref{section:geometry}, we prove some geometric results about $\genericFiber$,
and find an explicit set of divisors on $\genericFiberbar$.
In Section~\ref{section:picard}, we use these results to show that the geometric Picard number of $\genericFiber$ is $19$,
and we use the aforementioned set of explicit divisors on $\genericFiberbar$ to generate a rank 19
sublattice of the geometric Picard lattice of $\genericFiber$.
Finally, using a technique originating from  \cite{stoll-testa-10}, we prove that the two lattices coincide.

We then use the explicit description of $\Pic \Xetabar$ to prove
 a number of results about the geometry and the arithmetic of all the elements of the family $X$, as shown in Section~\ref{section:consequences}.
We also use Theorem~\ref{theorem:main} to obtain information about the elements
of a larger, in fact 4-dimensional, family of K3 surfaces (cf. Remark~\ref{r:4dimFamily}).

\section*{Acknowledgement}
We thank Anthony V\'arilly-Alvarado for bringing this problem to our attention, and
the Arizona Winter School for creating opportunities for research and providing an excellent platform for collaboration.
We thank
Noam Elkies,
Andreas-Stephan Elsenhans,
Eric Larson,
Ronald van Luijk,
Abraham Varghese,
and
Isabel Vogt
for helpful discussions;
and the anonymous referee for their detailed comments.

\section{Geometry}
\label{section:geometry}
In this section we investigate the geometry of $\genericFiber$.
First we show that $\genericFiber$ is a K3 surface,
then we exhibit an explicit elliptic fibration on it.
We then compute a subgroup of $\Aut \genericFiberbar$.
Finally, we write down a set of explicit divisors on $\genericFiberbar$;
these divisors play a crucial role in the proof of the main theorem.

Let us first fix notation. In this and also the following sections, if $Y$ is a scheme over a field $k$, we
denote by $Y_{\kbar}$ the base change of $Y$ to an algebraic closure of $k$.
For convenience we write $\genericFiberbar$ instead of $\genericFiber_{\overline{\QQ(t)}}$.
We denote by $\zeta_{12}$ a primitive $12$-th root of unity, and for $n =
3,4,6$ we define $\zeta_n$ as the $n$-th primitive root of unity given by
$\zeta_{12}^{12/n}$.
Furthermore, if $k$ is a field, we denote by $\PP_k$ the weighted projective space $\PP_k (1,1,1,3)$ over $k$, with coordinates $x,y,z,w$ of weight $1,1,1,3$, respectively.
We also denote by $\AA^1_k$ the affine line over $k$, with coordinate $t$.

We can now view the family $X$ as a threefold over $\QQ$ together with a fibration to the affine line.
Using the notation above, $X$ is the threefold
\begin{equation}\label{eq:family}
X\colon w^2=x^6+y^6+z^6+tx^2y^2z^2 \subseteq \PP_\QQ\times \AA^1_\QQ
\end{equation}
over $\QQ$. 
The fibration is the map $p\colon X \to \AA^1$ defined by $((x_0:y_0:z_0:w_0),t_0) \mapsto t_0$.

Let $t_0$ be a point in $\AA^1_k$ where $k$ is an algebraic extension of $\Q$.
The fiber $p^{-1} (t_0)\subset \PP_k\times \AA^1_k$ naturally embeds into $\PP_k$, 
and we denote its image  inside $\PP_k$ by $X_{t_0,k}$;
we also denote by $B_{t_0, k}$ the plane sextic curve in $\PP^2_k$ defined by
\begin{equation}\label{eq:branchlocus}
B_{t_0}\colon \XBeq{t_0} = 0.
\end{equation}

\begin{proposition}\label{prop:Xfibers}
Let $t_0$ be a point of $\AA^1_{\Qbar}$.
If $t_0^3\neq -27$, then $X_{{t_0},\Qbar}$ is a K3 surface.
If $t_0^3=-27$, then $X_{t_0,\Qbar}$ is birational to a K3 surface. Further, the surface
    $\genericFiber$ is a K3 surface over the function field $\QQ(t)$.
\end{proposition}
\begin{proof}
By definition, $X_{t_0,\Qbar}$ is the image inside $\PP_{\Qbar}$ of the fiber of $X$ above $t_0\in\AA^1_{\Qbar}$,
that is, the surface defined by the equation
$$
X_{t_0,\Qbar}\colon w^2=\XBeq{t_0} \subseteq \PP_{\overline{\QQ}}.
$$
Then $X_{t_0,\Qbar}$ is a double cover of $\PP^2_{\Qbar}$ ramified above the sextic curve $B_{t_0,\Qbar}$, defined in \eqref{eq:branchlocus}.
The curve $B_{t_0,\Qbar}$ admits singular points if and only if  $t_0^3=-27$.

Hence, if $t_0^3\neq -27$, the curve $B_{t_0,\Qbar}$ is smooth and so $X_{t_0,\Qbar}$ is a
    double cover of $\PP^2_\Qbar$ ramified above a smooth sextic curve and is thus a K3 surface.

Assume now that $t_0^3=-27$ or, equivalently, $t_0\in \{-3,-3\zeta_3, -3\zeta_3^2\}$.
One sees that $X_{-3,\Qbar}$ has twelve ordinary double points:
$(1: \zeta_6^j: \zeta_6^k : 0)$, where  $j+k \equiv 0 \bmod 3$.
For $i=1,2$, the map 
$(x:y:z:w)\mapsto (\zeta_3^i x : y:z:w)$
gives an isomorphism $X_{-3,\Qbar}\to X_{-3\zeta_3^i,\Qbar}$.
So also $X_{t_i,\Qbar}$ has twelve ordinary double points,
namely the points $(1 : \zeta_6^j : \zeta_6^k : 0)$,
such that $j+k \equiv i  \bmod 3$.
Recall that blow ups of points preserve the cohomological groups (cf. \cite[Proposition V.3.4]{hartshorne-77})
and that resolutions of ordinary double points preserve the canonical divisor (cf. \cite[Theorem III.7.2]{BHPV04}).
Hence, by blowing up the singular points of $X_{t_i,\Qbar}$ we obtain a K3 surface.

Finally, the surface $\genericFiber$ is the fiber of $X$ above the generic point of
    $\AA^1_\Q$ and so, by the first part of the proof, $\genericFiber$ is a K3 surface over the function field $\QQ (t)$.
\end{proof}

\begin{remark}\label{r:NoamFibration}
As noticed by Noam Elkies, 
$\genericFiber$ is isogenous to an elliptic K3 surface $\mathcal{Y}$. 
Both the isogeny and the elliptic fibration of $\mathcal{Y}$ can be explicitly written down, as follows.    

The composition of the Cremona transformation $(x:y:z:w)\mapsto (yz:xz:xy:w)$ 
with the map $(x:y:z:w)\mapsto (x^2:y^2:z^2:w)$
is a rational map from $\genericFiber$ to the surface defined by
$$
w^2 = (yz)^3+(xz)^3+(xy)^3+t(xyz)^2
$$       
inside $\PP_{\QQ (t)} (4,4,4,3) \cong \PP_{\QQ (t)}$.
Notice that this is a $4:1$ map not defined at $(1\colon0\colon0\colon 0)$, $(0\colon1\colon0\colon 0)$, and $(0\colon0\colon1\colon 0)$.

Set $z = r y$ to project along $(1\colon0\colon0\colon 0)$; for each $r$, we get 
\begin{equation*}
        w^2 = y^3((r^3 + 1) x^3 + r^2 t x^2 y + (r y)^3).
\end{equation*}
Now let us restrict to the affine patch given by $y\neq 0$ or,
equivalently, $y=1$.
Then we can write the above equation as 
\begin{equation*}
        w^2 = ((r^3 + 1) x^3 + r^2 t x^2  + r^3).
\end{equation*}
The $3:1$ map $\AA^2(x,w)\times \AA^1(r)\to \AA^2(u,v)\times \AA^1(s)$ defined by
$$((x,w),r)\mapsto ((x\,  r^4(r^3+1),w\, r^6(r^3+1)),r^3)$$
sends the variety defined by the above equation to the variety
\begin{equation}
        \label{equation:almostInose}
        \ellipticfibration: v^2 = u^3 + t s^2 u^2 + s^5 (1 + s)^2,
\end{equation}
a one parameter family of elliptic surfaces over the $s$-line. The composition
of these maps is a $12:1$ map from $\genericFiber$ to $\ellipticfibration$.

%Consider the map 
%    \[
%    (x\colon y\colon z\colon w)\mapsto(\sqrt{yz}\colon\sqrt{xz}\colon \sqrt{xy}\colon w),
%    \]
%    which is not defined at $(1\colon0\colon0\colon 0)$, $(0\colon1\colon0\colon 0)$, and $(0\colon0\colon1\colon 0)$.
%    Set $z = r y$ to project along $(1\colon0\colon0\colon 0)$; for each $r$, we get 
%    \begin{equation*}
%        w^2 = y^3((r^3 + 1) x^3 + r^2 t x^2 y + (r y)^3).
%    \end{equation*}
%    By setting $x = \frac{v y}{r^4 \left(r^3+1\right)}$, clearing square factors, and replacing $r^3$ by $s$,  we get the following equation:
%    \begin{equation}
%        \label{equation:almostInose}
%        \ellipticfibration: u^2 = v^3 + t s^2 v^2 + s^5 (1 + s)^2,
%    \end{equation}
%    a one parameter family of elliptic surfaces
%    over the $s$-line, with a $12:1$ map from $\genericFiber$ to $\ellipticfibration$.
\end{remark}

\begin{proposition}\label{prop:XKummer}
The surface $\genericFiber$ is isogenous to the Kummer 
surface associated to the abelian surface $E \times E$, where $E$ is an elliptic curve with $j$-invariant $-(4t)^3$.
For example, we can take $E$ to be
    \begin{equation*}
        E: y^2 + xy = x^3  + \frac{36}{1728 + (4 t)^3} x + \frac{1}{1728 + (4 t)^3}.
    \end{equation*} 
\end{proposition}
\begin{proof}
    Under the change of coordinates $v \mapsto v - s^2 t/3$ we can rewrite $\mathcal{Y}$ in \eqref{equation:almostInose} as
    \begin{equation*}
        u^2 = v^3 -  \frac{1}{3} s^4 t^2 v  + s^5 \left( s^2 + \frac{2}{27} s (27 + t^3) + 1\right),
    \end{equation*}
    which matches the Inose fibration
    \begin{equation}
        Z_{A, B}: y^2 = x^3 - 3 A s^4 x + s^5 (t^2 - 2 B s + 1),
    \end{equation}
    with $A = \frac{t^2}{9}$ and $B = \frac{1}{27} \left(-t^3-27\right)$.
    In \cite{inose-77}, Inose showed that $Z_{A, B}$ is isogenous to the Kummer surface associated to the product of two elliptic curves $E_1$ and $E_2$, where
    \begin{equation*}
        A^3 = j(E_1)j(E_2)/12^6, \quad B^2=\left(1 - j(E_1)/12^3\right)\left(1 - j(E_2)/12^3\right).
    \end{equation*}
    See also \cite[Section 13.5]{schuett-shioda-10}.
    In our case, we have $j(E_1) = j(E_2) = -(4t)^3$.
\end{proof}

\begin{remark}
This result came only after explicitly computing the Picard lattice.
Nevertheless, being independent of those computations, it can be used to more easily obtain some of our results. In particular, it implies Corollary~\ref{cor:NotIsotrivial}.
\end{remark}

\begin{corollary}\label{cor:NotIsotrivial}
The family $X$ is not isotrivial.
\end{corollary}
\begin{proof}
Since two elliptic curves are isomorphic if and only if they have the same $j$-invariant,
Proposition~\ref{prop:XKummer} implies that two fibers $X_{t_0}$ and $X_{t_1}$ are isomorphic if and only if $t_0^3=t_1^3$.
Hence, there exist smooth fibers that are not isomorphic.
\end{proof}

Automorphisms are important in order to understand the geometry of a surface.
Following \cite[Section 3.2]{Fes16}, we present a subgroup of 
$\Aut\genericFiberbar$.
We make use of this subgroup later, to find more divisors on $\genericFiberbar$ and generate a large sublattice of the Picard lattice.

Let $\sigma \in S_3$ be a permutation of the set $\{x,y,z\}$ and let $\psi_{\sigma}$ be the corresponding automorphism on $\PP_{\Qbar}$:
$$
   \psi_\sigma :
    (x\colon y\colon z\colon w)  \longmapsto (\sigma(x)\colon \sigma(y)\colon
    \sigma(z)\colon w).
$$
Further, for $i, j, k \in \ZZ/ 6 \ZZ$, such that $2(i+j+k)\equiv 0 \bmod 6$,  consider the following 
automorphism of $\PP_{\Qbar}$:
$$
    \psi_{i, j, k} :  
    (x\colon y\colon z\colon w) \longmapsto 
		\left( \zeta_6 ^i x \colon \zeta_6 ^j y \colon \zeta_6 ^k z\colon w\right).
$$

\begin{remark}\label{r:psi}
Since $w$ has weight 3, we have $(\zeta_6 ^2 x\colon \zeta_6 ^2 y\colon \zeta_6 ^2 z\colon w) = (x\colon y\colon z\colon w)$, and thus
$\psi_{i,j,k} = \psi_{i+\ell,j+\ell,k+\ell}$ for $\ell \equiv 0 \bmod 2$.
\end{remark}

Let $H_1$ denote the group formed by the automorphisms $\psi_\sigma$, for $\sigma\in S_3$;
let $H_2$ denote the group formed by the automorphisms $\psi_{i,j,k}$, for $i,
j, k \in \ZZ/ 6 \ZZ$, such that $2(i+j+k)\equiv 0 \bmod 6$;
finally, let $H$ denote the subgroup of $\Aut \PP _\Qbar$ generated by $H_1$ and $H_2$, that is
\begin{equation}\label{eq:H}
H:= \langle H_1, H_2 \rangle \subseteq \Aut \PP _\Qbar.
\end{equation}
Let $\alpha$ and $\beta$ be two elements of $H_1$ and $H_2$, respectively.
One can easily see that the automorphism given by 
$\alpha^{-1}\beta\alpha$ is an element of $H_2$.
We can then define an action of $H_1$ on $H_2$ by sending $(\alpha, \beta)\in H_1\times H_2 $ to $\alpha^{-1}\beta\alpha\in H_2$.
Let $H_1 \ltimes H_2 $ denote the semidirect product of $H_1$ and $H_2$, 
with $H_1$ acting on $H_2$ as described above.
It is easy to see that $H=H_1\ltimes H_2$.

The following results describe $H_1, H_2$ and $H$ as abstract groups. For a
positive integer $n$, let $\mu_n$ denote the multiplicative group of $n$-th
roots of unity.
Let $\Sigma\subset \mu_6^3$ be the subgroup of 
$\mu_6^3 = \mu_6 \times \mu_6 \times \mu_6$ defined by
$$
\Sigma := \{ (\zeta, \xi, \theta )\in \mu_6^3 \; : \; \zeta\xi\theta =\pm 1\}.
$$
\begin{remark}\label{r:Sigma}
The group $\Sigma$ is isomorphic to  
$\ZZ/6\ZZ \times \ZZ/6\ZZ \times 3\Z/6\Z$.
To see this, let $(\zeta, \xi, \theta )$ be an element of $\Sigma$.
Since $\zeta, \xi, \theta\in \mu_6$,
there are $i,j,k\in \{0,1,...,5\}$ such that 
$\zeta=\zeta_6^i, \xi=\zeta_6^j, \theta=\zeta_6^k$;
since $\zeta\xi\theta =\pm 1$,
we have that $i+j+k\in \{0,3\}$.
Then the map $\Sigma \to \ZZ/6\ZZ \times \ZZ/6\ZZ \times 3\Z/6\Z$ given by
$$
(\zeta, \xi, \theta )\to (i,j,i+j+k)
$$
is well defined and is in fact an isomorphism of groups.
\end{remark}
Let $\Delta\colon \mu_3 \hookrightarrow \mu_6^3$ 
 be the embedding defined by
 $$
 \Delta \colon \zeta \to (\zeta, \zeta, \zeta ).
 $$
 It is easy to see that the image of $\Delta$ is a subgroup of $\Sigma$.
 Let $N$ denote the quotient group 
 \begin{equation}\label{eq:N}
 N:=\Sigma/ \operatorname{im} (\Delta).
 \end{equation}

 \begin{remark}\label{r:H}
 As an easy exercise in group theory,
 one can show that
 the group $N$ is isomorphic to the group $(\Z/2\Z)^2\times \Z/6\Z$.
 \end{remark}

\begin{lemma}\label{l:G}
The following statements hold:
\begin{enumerate}
\item $H_1$ is isomorphic to the symmetric group $S_3$;
\item $H_2$ is isomorphic to the group $N$ defined in \eqref{eq:N};
\item $H$ is isomorphic to  $S_3 \ltimes N$,
where the action of $S_3$ on $N$ is given by permuting the coordinates of the elements of $N$.
\end{enumerate}
\end{lemma}
\begin{proof}
\begin{enumerate}[wide, labelwidth=!, labelindent=0pt]
\item Follows from the definition of $H_1$.
\item Let $(\zeta, \xi, \theta )$ be an element of $\Sigma$ and
let $i,j,k$ be defined as in Remark~\ref{r:Sigma}.
Then $i+j+k\in \{0,3\}$ or, equivalently,
$2(i+j+k)\equiv 0 \mod 6$.
We can then consider the map $\Sigma \to H_2$ given by
$$
(\zeta, \xi, \theta )\mapsto \psi_{i,j,k}.
$$
The map is clearly surjective;
by Remark~\ref{r:psi}, 
it follows that the kernel is the subgroup $\{ (0,0,0), (2,2,2), (4,4,4) \}$;
so 
$$H_2\cong \Sigma/ \{ (0,0,0), (2,2,2), (4,4,4) \} = N,$$
concluding the proof.
\item The statement follows by combining points~\textit{i)} and~\textit{ii)} and the observation that 
$H= H_1\ltimes H_2$.
\end{enumerate}
\end{proof}

\begin{lemma}
    \label{lemma:G isomorphism}
    The group $H$ injects into $\Aut\genericFiberbar$.
\end{lemma}
\begin{proof}
    First notice that all the elements of $H$ extend to automorphisms of  $\PP_{\overline{\QQ (t)}}$.
    The maps $\psi_{\sigma}$ and $\psi_{i,j,k}$ map $\genericFiber$ to 
itself, as the automorphisms fix the defining equation of $X$.
    For any given non-trivial $h \in H$, its fixed points are contained 
	in a hyperplane, hence it cannot act trivially on $\genericFiberbar$.
Hence, the group homomorphism $H\to \Aut\genericFiberbar$ given by first extending an automorphism $h\in H\subset \Aut \PP_{\overline{\QQ}}$ to $\PP_{\overline{\QQ (t)}}$ and then restricting it to $\overline{\genericFiber}$ is injective.
\end{proof}
In what follows we will make no distinction between the group $H$ and its image inside $\Aut\genericFiberbar$.

\begin{remark}
Since the elements of $H$ are defined over $\QQ (\zeta_6 ) \subset
    \overline{\QQ}$, using the same argument as in Lemma~\ref{lemma:G
    isomorphism}, one can show that for any $t_0 \in \AA^1$, the group $H$ is a subgroup of the automorphism group $\Aut X_{t_0,\Qbar}$.
\end{remark}

Finally, we produce a specific set of divisors on $\genericFiberbar$.
In Section~\ref{section:picard}, we show that this set of divisors, together with
a particular subgroup of isometries of $\Pic \genericFiberbar$, generates the full Picard lattice of $\genericFiberbar$.

As in \cite[Construction 4]{elsenhans-jahnel-08a}, to produce a set of divisors, 
we use the fact that $\genericFiber$ is a
double cover of $\PP^2_{\Q(t)}$ branched above the curve
\begin{equation*}
    {\branchCurve}\colon x^6+y^6+z^6+t x^2 y^2 z^2 = 0 \subset \PP^2_{\Q(t)}.
\end{equation*} 
We produce a set of divisors on $\genericFiberbar$ by searching for conics in $\PP^2_{\overline{\Q(t)}}$
that intersect $\branchCurve_{\overline{\Q(t)}}$ 
with even multiplicity everywhere; the pullback of such conics to the surface
$\genericFiberbar$ splits into two components.
For each of these conics, we take a component of the pullback. 
In this way, we find the set $\Omega$ of divisors on $\genericFiberbar$, given by
\begin{equation}
    \label{equation:Omega}
    \Omega = \left\{ B_1, B_2, B_3, B_4, B_5 \right\},
\end{equation}
where

\begin{align*}
B_1 & \colon \begin{cases}
x^2+y^2+\zeta_3 z^2&=0\\
w-\beta_1xyz &=0
\end{cases}
\\
B_2 &\colon \begin{cases}
x^2+\zeta_3y^2+\zeta_3^2 z^2&=0\\
w-\beta_0xyz &=0
\end{cases}
\\
B_3 &\colon \begin{cases}
2xy-c_1z^2 &= 0 \\
x^3-y^3-w &= 0
\end{cases}
\\
B_4 &\colon \begin{cases}
c_0\delta x^2- 2({9 c_0^2 + 3t c_0 - 2t^2})xy +2\delta y^2-\delta z^2&=0\\
(x^3+a_4x^2y+ b_4xy^2 + c_4y^3)({c_0^2c_1}+2)-2w &=0
\end{cases}
\\
B_5 &\colon \begin{cases}
a_5x^2 +c_5(y^2+z^2)+yz&=0\\
r_5x^3+ v_5xyz-w &=0.
\end{cases}
\end{align*}
Here, $\beta_i\in \overline{\QQ (t)}$ satisfy $\beta_i^2 = t+3\zeta_3^i$, for $i=0,1,2$; 
 the elements $c_0, c_1, c_2\in \overline{\QQ (t)}$ are the three roots of $x^3 + t x^2 + 4$;
 the element $\delta$ is $4\zeta_4\beta_0\beta_1\beta_2$; and the remaining
 constants are expressed in terms of $\beta_i, c_j$ and $\delta$ as follows:
\begin{align*}
a_4&=\frac{9c_0+6t}{4(t^3+27)}\delta,\\
b_4&=-c_0^2 - t c_0,\\
c_4&=\frac{18-3t^2c_0-3t c_0^2}{8(t^3+27)},\\
a_5&= \frac{\zeta_{12}(-\zeta_6+2)}{9}(\beta_0\beta_1+\beta_0\beta_2+\beta_1\beta_2+t),\\
c_5&=\frac{\zeta_{12}(\zeta_6-2)}{3},\\
r_5&=\frac{\zeta_{12}(\zeta_6-2)}{9}(2\beta_0\beta_1\beta_2+(2t-3)\beta_0+(2t-3\zeta_3)\beta_1+(2t+3\zeta_6)\beta_2),\\
v_5&= -\beta_0-\beta_1-\beta_2.
\end{align*}

\begin{remark}
In the rest of the paper, we abuse notation and use the symbols $B_i$, for
    $i=1,...,5$, to denote both the divisor $B_i$ as well as its class inside $\Pic \genericFiberbar$.
\end{remark}

\begin{remark}\label{r:GenusB}
By definition, $B_i$ is one of the two irreducible components on the pullback of
a certain smooth conic $C_i\subset \PP^2 _{\overline{\QQ(t)}}$ on $\genericFiberbar$. Thus $B_i$ is
    isomorphic to the conic $C_i$, and so $B_i$ has genus $0$.
\end{remark}

Note that the divisors in $\Omega$ are not defined over $\QQ (t)$.
Denote by $L$ the Galois closure of the 
minimal field extension of $\QQ(t)$ over which $\Omega$ is defined.
We refer to Appendix~\ref{section:galois} for the definition
of $L$ and the computation of the Galois group $\Gal (L/\QQ (t) )$, which we use
later on.

Let $S$ be a K3 surface defined over a field $k$ and let $K$ be a field extension of $k$.
Let $D$ be a class in $\Pic \overline{S}$.
In what follows, we say that $D$ \textit{can be defined over $K$} if $D$
contains a divisor on $\overline{S}$ whose defining equations have coefficients
in $K$.

\begin{remark}
    By exploiting the symmetry of the equation defining $\branchCurve$, it is easy to find
    the divisors $B_1, B_2, B_3$. However, finding the divisors $B_4$ and $B_5$
    is more challenging.
    To find them, we specialize the surface $\genericFiber$, find divisors on the specialization, and then try to lift them to $\genericFiber$.
    More explicitly, we choose $t_0 \in \ZZ$, and specialize the surface $\genericFiber$ to the surface $X_{t_0}$. 
    We then reduce $X_{t_0}$ modulo a prime of good reduction, $p$,
    to get the surface $X_{t_0 ,p}$ (see Section~\ref{section:picard} for more details).
    Then, with the help of a computer, we iterate over all conics over $\FF_p$ that are everywhere bi-tangent to the ramification locus of   $X_{t_0,p}$.
    We look for those conics whose components of the pull-back on $X_{t_0 ,p}$,
    together with the specializations of $B_1, B_2$ and $B_3$, 
    generate a lattice inside $\Pic \big(X_{t_0 ,p}\big) _{\overline{\FF_p}}$ with higher rank or smaller
    determinant than the lattice generated by $B_1, B_2$ and $B_3$ of $\genericFiber$.  
    For each such curve, we lift it to a number field. By repeating this process
    for different values of $t_0$, we are able to interpolate these divisors in term of $t$.
    For more details, see~\cite[Remarks 3.5.1 and 3.5.2]{Fes16}.
\end{remark}

\begin{remark}
There are many other ways to write down divisors on $\genericFiber$. For example, we could have looked for lines that are tri-tangent to $\mathcal{B}$.
    Unfortunately, such lines do not exist over $\QQ(t)$.
A Gr\"{o}bner bases computation shows that such lines only exist on the fibers above the zeros of the polynomial
\begin{equation}\label{eq:TriTL}
   t (t^3 + 5^3) (8t^3 + 33^3).
\end{equation}

Another possibility is to take advantage of the fact that $\genericFiber$ is a double cover of the degree-one del Pezzo surface 
\begin{equation*}
    Y:\ w = x^6 + y^6 + z^3 + t x^2 y^2 z\subset \PP_{\QQ(t)}(1,1,2,3).
\end{equation*}
It is well-known that $\Pic Y_{\overline{\QQ(t)}}$ is a rank 9 lattice, generated by the 240 exceptional curves on $Y_{\overline{\QQ(t)}}$.
These divisors do not play a relevant role in our approach, though this has
    been useful in other situations (e.g., \cite{kresch-tschinkel-04}).
\end{remark}

\section{The Picard lattice of $\genericFiberbar$}
\label{section:picard}
Even though there are some theoretical algorithms to compute the geometric Picard lattice of a K3 surface, none are computationally feasible.
For example, an effective version of the Kuga--Satake construction for degree-two K3 surfaces in \cite{hassett-kresch-tschikel-13} yields a theoretical algorithm, but there are no explicit examples of this.
Another algorithm is given in \cite[Section 8.6.]{poonen-testa-vanluijk-15}.
The paper \cite{charles-14} presents an algorithm to compute the rank of the
geometric Picard lattice, conditional on the Hodge conjecture for $X\times X$;
the algorithm computes a sublattice of finite index, which therefore has the
same rank.
These methods rely on finding linearly independent divisors on $X$. However, there is no known practical algorithm to do this on a K3 surface.

In this section, we use the results presented in Section~\ref{section:geometry} to explicitly compute the geometric Picard lattice of $\genericFiber$.
Using the divisors in the set $\Omega$, defined in ~\eqref{equation:Omega}, 
we explicitly compute a sublattice $\Lambda$ of $\Pic \genericFiberbar$.
Finally, we show that $\Lambda$ is in fact equal to $\Pic \genericFiberbar$.

We first compute the geometric Picard number.
\begin{proposition}\label{prop:PicNumber}
The surface $\genericFiber$ has geometric Picard number $\rho\big(\genericFiberbar\big)=19$.
\end{proposition}
\begin{proof}
    From Proposition~\ref{prop:XKummer}, we have that $\genericFiberbar$ is isomorphic to the Kummer surface associated to the square of an elliptic curve with $j$-invariant $-(4t)^3$.
    If $A$ is an abelian surface, with Kummer surface $\Kum(A)$, then we have
    $\rho(\Kum(A)) = 16 + \rho(A)$ (see \cite[Section
    17.1]{huybrechts-lectures}). Also, we have $\rho(A) = \dim (\End(A) \otimes
    \QQ)^\dagger$, where the superscript $\dagger$ denotes those endomorphisms
    that are invariant under the Rosati involution. For a product of elliptic
    curves, we have
    \begin{align*}
        (\End(E_1 \times E_2) \otimes \QQ)^\dagger \cong (\End(E_1) \otimes
        \QQ)^\dagger \times (\End(E_2) \otimes \QQ)^\dagger \times \Hom(E_1,
        E_2).
    \end{align*}
    See \cite[Section 21]{mumford-70} for more details. Therefore,
    \begin{align*}
        \rho\big(\genericFiberbar\big) &= 16 + \rho(E \times E )\\
                                       &= 18 + \rank \End( E )\\
                                       &\ge 19.
    \end{align*}
    In the other direction, from Corollary~\ref{cor:NotIsotrivial}, we have that $X$ is a $1$-dimensional
    non-isotrivial family of K3 surfaces.
Thus, from \cite[Corollary 3.2]{Dolgachev-96}, it follows that $\rho\big(\genericFiberbar\big)$ is at most $20-1=19$.
\end{proof}

Recall that $H$ is the subgroup of $\Aut \genericFiberbar$ defined in Section~\ref{section:geometry}.
The group $\Aut \genericFiberbar$ induces a natural action on the group of divisors on $\genericFiberbar$, so we define the set
\begin{equation}\label{eq:OrbitsOmega}
H\cdot \Omega = \{ hB\; |\; h\in H, \; B\in \Omega \; \},
\end{equation}
given by the union of the orbits of the elements of $\Omega$ under the action of $H\subset \Aut \genericFiberbar$.
Recall that all the divisors in $\Omega$ are defined over a certain Galois extension $L$ of $\QQ (t)$.
By construction, $\zeta_6$ lies in $L$ and so, 
since all the automorphisms in $H$ are defined over $\QQ(\zeta_6) (t)$, 
it follows that all the elements in $H\cdot \Omega$ are defined over $L$.

The divisors in $H\cdot \Omega$ generate a sublattice of $\Pic \genericFiberbar$, say 
\begin{equation}\label{eq:Lambda}
\Lambda:= \langle H\cdot \Omega \rangle \subseteq \Pic \genericFiberbar.
\end{equation} 
This sublattice can be explicitly computed, as shown by the following results.

\begin{remark}
    To show that $\rho \ge 19$ in Proposition~\ref{prop:PicNumber}, we could
    alternatively just compute the rank of the sublattice $\Lambda \subset \Pic
    \genericFiberbar$.
\end{remark}

\subsection{Computing the sublattice $\Lambda$.}
We first introduce some notation and state some preliminary results. 
For this, we follow \cite[Subsection 3.3.3]{Fes16}.

Let $t_0\in \ZZ$ be an integer and fix an integral model $\Xi_{t_0}$ for the surface $X_{t_0}$ (the fiber of $X$ above $t_0$).
Let $p\in\ZZ$ be a prime of good reduction 
for $\Xi_{t_0}$,
and let $\FF_p$ denote the field with $p$ elements.

Let $L_{t_0}$ be the number field obtained by 
specializing the field $L$  to $t=t_0$;
let $\calO_{t_0}$ denote the ring of integers of $L_{t_0}$,
and let $\pp$ be a prime of $\calO_{t_0}$ lying above $p$.
Let $\kappa (\pp)$ be the residue field $\calO_{t_0}/\pp$.
The field $\kappa (\pp)$ is isomorphic to $\FF_{p^m}$,
for some $m > 0$.

Let $X_{t_0,p}$ denote the reduction of $\Xi_{t_0}$ modulo $p$.
Let $B_{t_0,p}\subseteq \PP^2_{\FF_p}$ denote the branch locus of $X_{t_0,p}$.

Let $D$ be one of the divisors on $\genericFiber$ in $H\cdot \Omega$,
and let $\overline{D}$ denote its Zariski closure inside $X$.
Then $\overline{D}$ is a divisor on $X$.
We define $D_{t_0}$ to be the specialization of $\overline{D}$ at $t_0$; 
that is, the divisor on $X_{t_0}$ obtained by taking the fiber of $\overline{D}$ above $t_0$.

Note that for some divisors in $H\cdot \Omega$, there are values of $t_0 \in
\Qbar$ such that the divisor cannot be specialized to $t = t_0$; for example,
$B_4$ cannot be specialized to $t=-3$.
Assume that $D$ can be specialized to $t_0$ and 
that $\pp\in \calO_{t_0}$ is a prime of good reduction for $\Xi_{t_0}$.
Then let $\overline{D_{t_0}}$ be the Zariski closure of $D_{t_0}$ inside $\Xi_{t_0}$.
We define $D_{t_0,\pp}$ to be the reduction modulo $\pp$ of $\overline{D_{t_0}}$.
The curve $D_{t_0,\pp}$ is a divisor on $X_{t_0, \pp} = (X_{t_0 ,p})_{\kappa (\pp)}$.
Note that the procedure of going from a divisor on $\genericFiber$ to a divisor on
$X_{t_0, \pp}$ consists of a single step, repeated twice:
taking the closure of a divisor on the generic fiber of a family and specializing it to a special fiber.

\begin{lemma}\label{lemma:specialization}
Using the same notation as above,
and assuming that $D$ and $D'$ are two divisors on $\genericFiber$ that can be specialized to $t_0\in \ZZ$,
we have the following equality of intersection numbers:
\[
D\cdot D' = D_{t_0,\pp}\cdot D'_{t_0,\pp}
\]
\end{lemma}
\begin{proof}
Using \cite[Proposition 3.6]{maulik-poonen} and  \cite[Theorem
    1.4]{elsenhans-jahnel-11b}, one immediately gets that the reduction of 
$\genericFiber$ to $X_{t_0,p}$ induces an embedding
$\Pic \genericFiberbar \hookrightarrow \Pic (X_{t_0,\pp})_{\overline{{\kappa (\pp)}}}$
that is injective, compatible with the intersection product, and has torsion-free cokernel.
The result follows.

See \cite[Proposition 17.2.10, Remark 17.2.11]{huybrechts-lectures}, and
    \cite[Proposition 6.2]{vanluijk-07a} for more details.
\end{proof}

\begin{computation}\label{computation:Lambda}
The lattice $\Lambda$ is isometric to the lattice
\begin{equation*}
     U \oplus E_8 ( -1) \oplus A_5 (-1) \oplus A_2 (-1) 
     \oplus A_2(-4).
\end{equation*}
    In particular, $\Lambda$ has rank 19, determinant $2^5 \cdot 3^3$, signature $(1,18)$ and discriminant group isomorphic to $\ZZ/6 \ZZ \times \left(\ZZ/12\ZZ\right)^2$.
\end{computation}
See Appendix~\ref{section:App:Lattices} for the definition of the lattices $U,E_8(-1), A_5(-1), A_2(-1),$  $A_2(-4)$, the direct sum of lattices, and the discriminant group of a lattice.
\begin{proof}
  The main step is to compute the intersection matrix for the generators of
    $\Lambda$; that is, to calculate the intersection numbers $D\cdot D$ for all pairs $D, D'$ in the set $H\cdot \Omega$.

For self intersection numbers, we use the adjunction formula:
    \begin{equation*}
        D\cdot D = 2g(D) - 2 = -2, \quad \text{ for }D \in  H \cdot \Omega,
    \end{equation*}
as      $D$ is isomorphic to a plane conic,  and so it has genus zero.
    
For pairs of distinct $D, D'$ in $H\cdot \Omega$, 
directly computing the intersection number $D\cdot D'$ over an algebraic extension of a function field can be computationally very expensive.
Instead, using Lemma~\ref{lemma:specialization}, we reduce all the computations to computations over finite fields.
Fix an integer $t_0\in\ZZ$, 
and an integral model for $X_{t_0}$.
Let $p$ be a prime of good reduction for the fixed integral model of $X_{t_0}$
and, 
recalling the notation introduced before,
let $L_{t_0}$ be the specialization of $L$ to $t_0$,
$\calO_{t_0}$ be the ring of integers of $L_{t_0}$
and $\pp$ be a prime of $\calO_{t_0}$ lying above $p$.
Using Lemma~\ref{lemma:specialization}, 
if $D,D'$ are two divisors on $\genericFiber$,
then $D\cdot D' = D_{t_0,\pp} \cdot D'_{t_0,\pp}$.
Since all divisors $D\in H\cdot\Omega$ are defined over $L$,
all the divisors $D_{t_0,\pp}$ are defined over the finite field $\FF_{p^m} \cong \kappa(\pp)$, 
for some $m >0$.

If $D_{t_0,\pp}$ and  $D'_{t_0,\pp}$ have no common components,
then the intersection $D_{t_0,\pp}\cap D'_{t_0,\pp}$ is a zero-dimensional scheme over $\FF_{p^m}$.
Using {\sc Magma} it  is possible to compute its degree.
Since we are considering divisors on a smooth surface,
the degree of the zero-dimensional scheme given by the intersection of the two divisors
equals the sum of the intersection multiplicities of the points of intersection of the two divisors
(see~\cite[A.2.3]{HS00}),
and so the degree of  $D_{t_0,\pp}\cap D'_{t_0,\pp}$ is the intersection number $D_{t_0,\pp}\cdot D'_{t_0,\pp}=D\cdot D'$.
In this way, we get the intersection matrix of the lattice $\Lambda$ generated by  $D\in H\cdot \Omega$.

 In our computations, we use $t_0=7$ and $p = 79$ and work over
    $\FF_{79^2}$.
    We then compute an integral basis for $\Lambda$, and compute the intersection matrix of the basis; this is the Gram matrix of $\Lambda$.
    Finally, the rank, signature, and discriminant of $\Lambda$ are the rank, signature, and determinant of the Gram matrix, respectively.
    One also computes the discriminant group of $\Lambda$ from the Gram matrix, as described in Appendix~\ref{section:App:Lattices}.
These quantities agree with those given in the statement of the result.

Using these data, the first statement of the result immediately follows from \cite[Corollary
    1.13.3]{nikulin-80} (see Appendix~\ref{section:App:Lattices}), noticing that
    $\Lambda$ and the lattice in the statement are both even and indefinite, and
    have the same rank, discriminant, signature, discriminant form, and the same
    number of generators for the discriminant group; that is, $3<17=19-2$.
    For the {\sc Magma} code with the effective computations, see~\cite{BCFNW17}.
\end{proof}

Before stating a corollary of Computations~\ref{computation:Lambda}, recall that
a class $D\in \Pic \genericFiberbar$ is said to be defined over an extension $K/ \QQ
(t)$ if $D$ contains a divisor whose defining equations have coefficients in
$K$. We also introduce the notation $\OO \left( \Pic \genericFiberbar \right)$ for the
group of isometries of $\Pic \genericFiberbar$.

\begin{corollary}\label{cor:PicLatt}
    The following statements hold.
    \begin{enumerate}
        \item The lattice $\Lambda$ is a finite index sublattice of $\Pic \genericFiberbar$.
        \item Every class in $\Pic \genericFiberbar$ is defined over $L$.
        \item The Galois group $\Gal (L/\QQ (t) )$ injects into $\OO \big(\Pic \genericFiberbar \big)$.
        \item The group $H$ injects into $\OO \left(\Pic \genericFiberbar \right)$.
    \end{enumerate}
\end{corollary}
\begin{proof}
    \begin{enumerate}[wide, labelwidth=!, labelindent=0pt]
        \item  This is immediate, since $\Lambda$ is a sublattice of $\Pic \genericFiberbar$ by construction,
            and from Proposition~\ref{prop:PicNumber} and Corollary~\ref{cor:PicLatt} we have that $\Pic \genericFiberbar$ and $\Lambda$ both have rank $19$.
        \item 
            We have already seen that all the elements in $H\cdot \Omega$ are defined over $L$. 
            Since these elements generate the lattice $\Lambda$, 
            we have that all the elements of $\Lambda$ are defined over $L$.
            Now let $D$ be a class in $\Pic \genericFiberbar$.
            By point (1) we have that $\Lambda$ has finite index, say $d$, inside $\Pic \genericFiberbar$. 
            Then $dD$ is an element of $\Lambda$, hence it can be defined over $L$. Since the action of the Galois group on $\Pic \genericFiberbar$ is linear and $\Pic \genericFiberbar$ is torsion free, it follows that $D$ can be defined over $L$.
        \item
            The Galois group $G_t=\Gal (\overline{\QQ (t)}/ \QQ (t) )$ naturally acts on $\Pic \genericFiberbar$; that is, 
            there is a map $G_t \to \OO (\Pic \genericFiberbar )$.
            We have seen that all the elements of $\Pic \genericFiberbar$ are defined over $L$, 
            hence this action factors through the Galois group $\Gal (L/\QQ (t) )$.
            More precisely, we have an exact sequence 
            $$
            0\to H_t \to  G_t \to \OO (\Pic \genericFiberbar),
            $$
            where $H_t$ is the kernel of the map $G_t \to \OO (\Pic \genericFiberbar )$.
            As the elements of $\Pic \genericFiberbar$ are defined over $L$, we have that $H_t$
            contains $\Gal (\overline{\QQ (t)}/ L)$.
            Having an explicit description of $\Gal (L/\QQ (t) )$ and $\Lambda$, one can easily see that no element of $\Gal (L/\QQ (t) )$ acts as the identity on $\Lambda \subset \Pic \genericFiberbar$, 
            hence no element of $\Gal (L/\QQ (t) )$ acts as the identity on $\Pic \genericFiberbar$
            (see Appendix~\ref{section:galois}  for the explicit description of $\Gal (L/\QQ (t) )$ and use the explicit divisors given in Section~\ref{section:geometry} to see that for every element of $\Gal (L/\QQ (t) )$ there is at least one element in that list that is not fixed).
            It follows that the kernel $H_t$ is exactly the Galois group $\Gal (\overline{\QQ (t)}/ L)$.
            Therefore, 
            $$
            \Gal (L/\QQ (t) ) \cong G_t/ \Gal (\overline{\QQ (t)}/ L) = G_t/H_t
            $$
            injects into $\OO \left(\Pic \genericFiberbar \right)$,
            proving the statement.
        \item
            Recall that $H$ is the subgroup of $\Aut \genericFiberbar$ defined in~\eqref{eq:H}.
            From Proposition~\ref{prop:PicNumber}, we know that $\Pic \genericFiberbar$ has Picard number $19$;
            from Computation~\ref{computation:Lambda} we know that $\Lambda$ has determinant $2^5\cdot 3^3$;
            from point (1), we know that $\Lambda$ is a finite index sublattice of $\Pic \genericFiberbar$.
            It follows that $\Pic \genericFiberbar$ has rank $19$ and that its determinant is not
            a power of $2$; in particular, the determinant is divisible by $3$.
            The statement now immediately follows from \cite[Proposition 1.2.47]{Fes16}.
    \end{enumerate}
\end{proof}

From Corollary~\ref{cor:PicLatt}, we have that both groups $\Gal (L/\QQ (t) )$ and $H$ embed into $\OO (\Pic \genericFiber)$,
so we can define the group
\begin{equation}\label{eq:G}
G:= \langle \Gal (L/\QQ (t) ), H \rangle \subseteq \OO (\Pic \genericFiberbar),
\end{equation}
the subgroup of $\OO (\Pic \genericFiberbar)$ generated by $\Gal (L/\QQ (t) )$ and
$H$.

\subsection{Proving $\Lambda=\Pic \genericFiberbar$.}
Recall that $\Lambda$ is the sublattice of $\Pic \genericFiberbar$ generated by the orbits of the 
divisors $\Omega$ under the group $H$ (cf.~\eqref{eq:Lambda}).

We now show that $\Lambda$ is in fact the whole geometric Picard lattice of $\genericFiber$, by showing that the quotient $\Pic \genericFiberbar / \Lambda$ is trivial.
From Corollary~\ref{cor:PicLatt}, we know that $\Lambda$ is a finite index sublattice of $\Pic \genericFiberbar$, and consequently that $\Pic \genericFiberbar / \Lambda$ is a finite abelian group.
This group is trivial if and only if there is no element of order $p$, for any
prime $p$.
An element of order $p$ corresponds to a nontrivial element of the kernel of the natural map $\Lambda / p \Lambda \rightarrow \Pic \genericFiberbar / p \Pic \genericFiberbar$.
Denote the kernel of this map by $\Lambda_p$.
To complete the proof, it suffices to show that $\Lambda_p$ is trivial for every prime $p$.

\begin{remark}\label{r:TrivialPrimes}
We can eliminate all but two primes by considering the possible order of $\Pic \genericFiberbar / \Lambda$.
Indeed, $\Lambda_p \not= 0$ if and only if there is an element of order $p$ in $\Pic \genericFiberbar / \Lambda$, which implies that $p$ divides $\#\left(\Pic \genericFiberbar / \Lambda\right) = [\Pic \genericFiberbar \colon \Lambda]$.
If $L$ and $L'$ are lattices of the same rank such that $L' \subseteq L$, then $\disc L' = [L \colon L']^2 \disc L$.
From Computation~\ref{computation:Lambda}, we have $\disc \Lambda = 2^5 \cdot 3^3$; thus $[\Pic \genericFiberbar \colon \Lambda] \mid 2^2 \cdot 3$.
Consequently, $\Lambda_p = 0$ for $p \not\in \{2, 3\}$.
\end{remark}

We are left to show that $\Lambda_2$ and $\Lambda_3$ are both zero.
For this, we use the technique in \cite{stoll-testa-10}, which is also used in \cite{VarillyAlvaradoViray}:
since we cannot compute $\Lambda_p$ directly, 
we find a subset of $\Lambda / p \Lambda$ that contains $\Lambda_p$ and 
that can be explicitly computed, 
then we show that none of the elements of this subset can be an element of $\Lambda_p$.
If $D$ is an element of $\Lambda$, we write $[D]_\Lambda=D \bmod p\Lambda$ to denote the class of $D$ inside $\Lambda/p\Lambda$.
The intersection pairing on $\Lambda$ induces 
another symmetric pairing on $\Lambda/p\Lambda$,
\begin{align*}
\Lambda/p\Lambda \times \Lambda/p\Lambda &\rightarrow \ZZ / p \ZZ \\
    [D]_\Lambda \cdot [D']_\Lambda &\mapsto D\cdot   D' \bmod p.
\end{align*}
Let $k_p$ denote the left kernel of this composition,
and let $M_p$ denote the subset of $k_p$ given by
\begin{equation}
M_p:=\{ [D]_\Lambda \in k_p \; | \; [D]_\Lambda ^2 \equiv 0 \bmod 2p^2 \; \}.
\end{equation}

\begin{lemma}\label{l:MpInv}
The subset $M_p$ contains $\Lambda_p$ and it is fixed by all the isometries of $\Lambda_p$.
\end{lemma}
\begin{proof}
This proof is purely lattice-theoretic.
For example, see~\cite[Lemma 1.1.23]{Fes16}.
\end{proof}

\begin{lemma}\label{l:GLambda}
The group $G$ defined in~\eqref{eq:G} injects into $\OO (\Lambda )$.
\end{lemma}
\begin{proof}
By definition, $G$ is a group of isometries of $\Pic \genericFiberbar$.
Using the explicit description of the elements of $G$ and $\Lambda$,
one can check that $G$ preserves $\Lambda$ and that no element of $G$ acts as the identity on $\Lambda$.
The statement follows.
\end{proof}

We say that a class $E\in\Pic \genericFiberbar$ is \emph{effective} if it contains an effective divisor.
\begin{lemma}\label{lemma:riemannroch}
Let $S$ be a K3 surface.
If $E\in \Pic S$ is a divisor class on $S$ such that $E^2 = -2$,
 then either $E$ or $-E$ is effective.
\end{lemma}
\begin{proof}
    This is a well-known general result, see for example \cite[Lemma 1.2.35]{Fes16} or \cite[Section 1.2.3]{huybrechts-lectures}.
\end{proof}

We are now ready to prove our main result.

\begin{lemma}\label{l:p23}
The kernel $\Lambda_p$ is trivial for every prime $p$.
\end{lemma}
\begin{proof}
By Remark~\ref{r:TrivialPrimes}, the statement holds for all primes $p\neq 2,3$,
so it suffices to prove the statement for $p=2$ and $p=3$.
For this, we follow the proof of~\cite[Theorem 3.1.4]{Fes16}.

Let $p$ be equal to $2$ or $3$.
From Lemma~\ref{l:GLambda}, we know that $G$ acts on $\Lambda$.
Since the $\FF_p$-vector space $\Lambda_p$ is the kernel of a $G$-equivariant homomorphism, it is $G$-invariant.
So if an element $[x]_\Lambda$ is in $\Lambda_p$, 
its whole $G$-orbit, $G\cdot [x]_\Lambda$, is contained in $\Lambda_p$.
Since the discriminant of $\Lambda$ is $2^5\, 3^3$, 
the $\FF_p$-vector space $\Lambda_p$ can have dimension  
at most~$2$ and $1$,
for $p=2,3$  respectively.
Since $\Lambda_p$ is stable under the action of $G$,
it follows that the $G$-orbit of every element in $\Lambda_p$ spans an $\FF_p$-vector space of dimension at most $2$ or $1$, for $p=2,3$ respectively.

Analogous statements hold if we consider the action of just $H$ (cf.~\eqref{eq:H}),
instead of the whole of $G$.

Let $p=2$.
Using {\sc Magma}, we explicitly compute the subset  $M_p$.
There is only one non-trivial $H$-orbit inside $M_p$ spanning a vector subspace
    of dimension at most~2.
Let $W$ denote this subspace.
The subspace $W$ has dimension $2$, 
and it admits a basis $\{w_1, w_2\}$ 
such that
\begin{align*}
w_1 &= [E_1]_\Lambda \\
w_2 &= [E_2]_\Lambda,
\end{align*}
where the divisors are $E_1:=\psi_{0,3,0}B_3 -B_3$ and $E_2:=\tau_2^2\psi_{(x,y)} (\psi_{0,3,0} B_4 - B_4)$,
and $\tau_2$ denotes the automorphism in $\Gal (L/\QQ (t) )$ defined in Appendix~\ref{section:galois} (cf. Table~\ref{table:GalAut}).
Using the same technique as in Computations~\ref{computation:Lambda},
one checks that $E_1^2=E_2^2=-8$.
Assume $w_1$  is an element of $\Lambda_p$,
then  $E_1$ is an element of $\Lambda$ 
that is $2$-divisible in $\Pic \genericFiberbar$,
say $E_1=2C$, 
for some $C\in \Pic \genericFiberbar$.
Since $E_1^2=-8$,
the class $C$ is a $-2$-class,
and
then either $C'$ or $-C'$ is effective (cf. Lemma~\ref{lemma:riemannroch}). 
By construction, $E_1=E_{1,1}-E_{1,2}$,
where $E_{1,1}=\psi_{0,3,0}B_4$ and
$E_{1,2}=D_4$.
Note that both $E_{1,1}$ and $E_{1,2}$ are elements of $G\cdot \Omega$, and so
    are isomorphic to plane conics; hence, $E_{1,1}^2=E_{1,2}^2=-2$.
Let $L$ be the hyperplane class in $\Pic \genericFiberbar$, and notice that it is
    ample (in fact, $3L$ is very ample).
Since $E_1=2C$, with $C$ a $-2$-class, and $L$ is ample, we have that the intersection number $H\cdot E_1=2 H\cdot C $ is either positive or negative (according to whether $C$ or $-C$ is effective);
on the other hand,  $E_1=E_{1,1}-E_{1,2}$, and so $L\cdot E_1=H\cdot E_{1,1} - L\cdot E_{1,2} = 2-2 = 0$, yielding a contradiction.
Therefore $E_1$ cannot be $2$-divisible.
The same argument holds for $E_2$, as well as for any other element of $W$,
since the orbit of every element of $W$ spans the whole of $W$; indeed, in $M_p$
    there are no $1$-dimensional subspaces generated by $H$-orbits.
Therefore $\Lambda_2$ is trivial.

Let $p=3$.
As before, we compute the subset $M_p$ using {\sc Magma}.
Among the non-trivial vectors in $M_p$, we look for those whose orbit under $G$  spans a $1$-dimensional $\FF_3$-vector space.
It turns out that there are no such vectors. Thus $\Lambda_3$ is also trivial, completing the proof.

See~\cite{BCFNW17} for the {\sc Magma} code used to perform the computations.
\end{proof}

As noted before, Lemma~\ref{l:p23} finally shows that
$$
\Pic \genericFiberbar = \Lambda.
$$
This, together with Computations~\ref{computation:Lambda}, completes the proof of Theorem~\ref{theorem:main}.

\section{Consequences}
\label{section:consequences}
The explicit description of $\Pic \genericFiberbar$ enables us to understand some arithmetic properties of the family $X$.
For example, we can compute the Galois structure of $\Pic \genericFiberbar$ and deduce the following result.

\begin{computation}\label{computation:cohomology}
    Considering the action of $\Gal(L/\QQ(t))$ on $\Pic \genericFiberbar$, the following statements hold.
\begin{enumerate}
    \item $\coho{0}\big(\Gal(L / \QQ(t)),\Pic \genericFiberbar\big)$ is isomorphic to $\ZZ$, and it is generated by the class of the hyperplane section of $\genericFiber$.
    \item $\coho{1}\big(\Gal(L / \QQ(t)),\Pic \genericFiberbar\big)$ is isomorphic to $( \ZZ / 2\ZZ)^3$.
\item  $\coho{2}(\Gal(L / \QQ(t)),\Pic \genericFiberbar\big)$ is isomorphic to $( \ZZ / 2\ZZ)^{10}$.
\item Further, for every non-trivial subgroup $M \subset \Gal(L / \QQ(t))$, we have 
\[
    \coho{1}(M, \Pic \genericFiber) \cong ( \ZZ / 2\ZZ)^i
\]
 with $i \in \{0,1,2,3,4,5,6,8,10,12\}$.
 \item There are 49 normal subgroups $N$ of $\Gal (L / \QQ(t))$ for which $\coho{1}( N, \Pic (\genericFiberbar))$ is trivial.
 \item There are 47 normal subgroups $N$ of $\Gal (L / \QQ(t))$ for which $\coho{1}( N, \Pic (\genericFiberbar))$ is non-trivial.
\end{enumerate}
\end{computation}
\begin{proof}
By direct computations using {\sc Magma}; see~\cite{BCFNW17} for the code.
\end{proof}
\begin{remark}
Computation~\ref{computation:cohomology} can be useful for studying the
    existence of a Brauer-Manin obstruction on $\genericFiber$.
In fact, the Hochschild--Serre spectral sequence implies the following isomorphism:
\begin{equation*}
    \label{equation:Br1/Br0-quotient}
    \frac{\Br_1 \genericFiber}{\Br_0 \genericFiber} \overset{\sim}{\longrightarrow} \coho{1}\bigl(
\Gal\bigl(L/\QQ (t)\bigr), \Pic \genericFiberbar \bigr),
\end{equation*}
where $\Br_0 \genericFiber$ and $\Br_1 \genericFiber$ are the constant and the algebraic Brauer group of $\genericFiber$, respectively.
\end{remark}

Theorem~\ref{theorem:main} not only gives us information about the arithmetic of $\genericFiber$, but also of the other fibers of $X$, 
as shown by the following results.

\begin{corollary}\label{c:PicNumbFibers}
Any smooth fiber $X_{t_0}$ of the family $X$ has geometric Picard number at least $19$.
If the geometric Picard number of $X_{t_0}$ is exactly $19$, 
then $\Pic X_{t_0, \Qbar}$ is isometric to $\Lambda$.
The geometric Picard number of $X_{t_0}$ is $20$ if and only if $-(4 t_0)^3$ is the $j$-invariant of an elliptic curve with complex multiplication.
\end{corollary}
\begin{proof}
Theorem~\ref{theorem:main} tells us that $\Pic \genericFiberbar=\Lambda$.
From \cite[Proposition 3.6]{maulik-poonen} we have that there is an embedding of
$\Pic \genericFiberbar$ inside $\Pic X_{t_0,\Qbar}$ that respects the intersection
    pairing and has torsion free co-kernel.
The existence of the embedding implies that $\Pic X_{t_0,\Qbar}$ has rank at least $19$.
Since the embedding has torsion free co-kernel, if $\Pic X_{t_0, \Qbar}$ has rank exactly $19$, then the image of $\Pic \genericFiberbar$ is the whole geometric Picard lattice of $X_{t_0}$.
Finally, the last statement follows immediately from Proposition~\ref{prop:XKummer} and~\cite[12.(26)]{schuett-shioda-10}.
\end{proof}

\begin{remark}
It follows that there are only finitely many rational values of $t_0$ for which the Picard number of $X_{t_0}$ is $20$, namely
$$
160080,\;
1320, \;
240,\;
24,\;
8,\;
15/4,\;
0,\;
-3,\;
-5,\;
-33/2,\;
-255/4.
$$ 
Note that $0, -5,$ and $-33/2$ are the rational roots of the
    polynomial~\eqref{eq:TriTL} from Section~\ref{section:geometry} whose roots
    are the values of $t_0$ for which $X_{t_0}$ admits tri-tangent lines.
\end{remark}

\begin{corollary}\label{c:PotRatPts}
    Let $t_0 \in \Qbar$ be an algebraic number such that the fiber $X_{t_0}$ is smooth.
Then $X_{t_0}$ is a K3 surface defined over the number field $\QQ (t_0)$ and it has potentially dense rational points.
\end{corollary}
\begin{proof}
Since $X_{t_0}$ is smooth, it is a K3 surface, by Proposition~\ref{prop:Xfibers};
since $t_0$ is algebraic over $\QQ$, the field $\QQ (t_0)$ is a number field,
 and so the surface $X_{t_0}$ can be defined over the number field $\QQ (t_0)$.
From Corollary~\ref{c:PicNumbFibers} we have that $X_{t_0}$ has geometric Picard number at least $19$.
Then the statement immediately follows from~\cite[Theorem 1.1]{bogomolov-tschinkel}.
\end{proof}

\begin{remark}\label{r:4dimFamily}
Let $a,b,c,d\in \overline{\QQ}$ be four algebraic numbers, and
let $X_{a,b,c,d}$ be the surface defined by 
\begin{equation}\label{eq:4dimFamily}
X_{a,b,c,d}\colon w^2 = ax^6 +by^6 +cz^6 +dx^2y^2z^2\subset \PP_\Qbar.
\end{equation}
    Assume $X_{a,b,c,d}$ is smooth. 
%Then $X_{a,b,c,d}$ is a K3 surface (cf. Proposition~\ref{prop:Xfibers}).

    It is easy to see that $X_{a,b,c,d}$ is isomorphic, over $\overline{\QQ}$, to $X_{e, \Qbar}$,
the fiber of the family $X$ above the point $e=d\cdot \varepsilon$, where $\varepsilon$ is a third-root of the product $a \cdot b \cdot c$.

This implies that the geometric Picard lattice of $X_{a,b,c,d}$ is isometric to $\Pic X_{e, \Qbar}$,
making it possible to use Corollaries~\ref{c:PicNumbFibers} and~\ref{c:PotRatPts} for K3 surfaces with defining equation as in~\eqref{eq:4dimFamily}.
\end{remark}

\clearpage
\appendix

\section{Exposition on Lattices}\label{section:App:Lattices}
We briefly review the lattice theory that we use in Section~\ref{section:picard}, and refer the reader to \cite{nikulin-80} for more details.

In this article a \emph{lattice} is a free abelian group, $L$, of finite rank equipped with a symmetric, non-degenerate, bilinear form $\left\langle \,,\,\right\rangle :L\times L\to\mathbb{Z}$.
We define the signature of $L$ to be the signature of the extension to $\RR$ of its bilinear form.
We say that $L$ is \emph{positive definite} if it has signature $(b_{+},0)$, \emph{negative definite} if it has signature $(0,b_{-})$, and \emph{indefinite} otherwise.
We say that $L$ is \emph{even} if $\left\langle x,x\right\rangle \in2\mathbb{Z}$ for all $x\in L$.
Let $\{e_{i}\}$ be a basis for $L$; then the \emph{Gram matrix} of $L$ with respect to $\left\{ e_{i}\right\} $ is the matrix $\left(\left\langle e_{i},e_{j}\right\rangle \right)_{i,j}$.
The \emph{discriminant} of $L$, denoted $\mathrm{Disc}(L)$, is the
determinant of the Gram matrix with respect to some basis; it is independent of the choice of basis.

We define $A_{n}(m)$ and $E_{8}(m)$ to be the lattices 
obtained  from the root lattices $A_n$ and $E_8$ (cf. \cite[Sections 4.6.1 and 4.8.1]{CS99}), respectively, by multiplying  their quadratic form by $m$.

Let $L_1$ and $L_2$ denote two lattices, with bases $\{e_i\}$ and $\{f_i\}$, respectively.
Then $L_1 \oplus L_2$ denotes the lattice with basis $\{e_i\} \cup \{f_i\}$ and bilinear form extending that on $L_1$ and $L_2$ such that $\langle e_i, f_j \rangle = 0$ for all $i, j$.
If $L_1 \subseteq L_2$ and $\rank(L_1) = \rank(L_2)$, then we say that $L_1$ is a \emph{full rank sublattice} of $L_2$.
In this situation, $\mathrm{Disc}(L_{1})/\mathrm{Disc}(L_{2})=[L_{2}:L_{1}]^{2}$. 

For any lattice $L$ we define the \emph{dual lattice}
\[
L^{*}:=\Hom(L,\mathbb{Z})\cong\left\{ x\in L\otimes\mathbb{Q}:\; \forall y\in L \; \left\langle x,y\right\rangle \in\mathbb{Z}\,\right\} .
\]

The \emph{discriminant group} of a lattice $L$ is the finite abelian group $A_{L}:=L^{*}/L$, and we denote by $\ell(A_{L})$ the minimal number of generators of $A_{L}$. 
The discriminant group comes equipped with a bilinear form, $b_{L}:A_{L}\times A_{L}\to\QQ/\ZZ$, defined by $b_{L}(x+L,y+L) = \left\langle x,y\right\rangle \bmod\mathbb{Z}$.

If $L$ is an even lattice, we define the \emph{discriminant form},
\begin{align*}
q_{L}:A_{L} &\to \mathbb{Q}/2\mathbb{Z} \\
x+L &\mapsto \left\langle x,x\right\rangle \bmod2\mathbb{Z}.
\end{align*}

We use the following theorem in Computation~\ref{computation:Lambda} to identify the lattice $\Lambda$.
\begin{theorem}[{Nikulin \cite[Cor. 1.13.3]{nikulin1980integral}}]
If a lattice $L$ is even and indefinite with $\mathrm{rank(}L)>\ell(A_{L})+2$, then $L$ is determined up to isometry by its rank, signature and discriminant form.
\end{theorem}

\clearpage

\section{The Galois group in the Generic Case}\label{section:galois}
Let $K$ be the function field $\QQ(t)$.
In Section~\ref{section:geometry}
we fixed an algebraic closure $\Kbar=\overline{\QQ (t)}$ of $\QQ (t)$, and
defined some  elements in this algebraic closure.
Recall that:
$\zeta_{12}$ is defined to be a primitive $12$-th root of unity;
$\zeta_n$ is defined to be the primitive $n$-th root of unity given by $\zeta_{12}^{12/n}$, for $n=3,4,6$;
 the elements $\beta_i$ are such that $\beta_i^2=t+3\zeta_3^i$, for $i=0,1,2$;
 the elements $c_0,c_1,c_2$ are the roots of the polynomial 
 \begin{equation}\label{eq:h}
 h(x):=x^3+tx^2+4\in\QQ (t) [x];
 \end{equation}
 the element $\delta$ is the product $4\zeta_4\beta_0\beta_1\beta_2$.
We also defined the field $L$ to be the Galois closure of the smallest extension of $\QQ (t)$ containing all the elements defined above.
In this section we explicitly describe $L$, as well as the Galois group $\Gal(L/\QQ (t))$, proving the following theorem.
\begin{theorem}\label{thm:GaloisMain}
The field $L$ is the field $K (\zeta_{12}, \beta_0, \beta_1, \beta_2, c_0)$.
It is a  Galois extension of degree $2^5\cdot 3$ and its Galois group is isomorphic to the group
$$
S_3 \times \Z/2\Z \times D_4.
$$
\end{theorem}
In the statement of the theorem,
$S_3,\ \Z/2\Z,$ and $D_4$ denote 
the permutation group of a set with three elements, 
the cyclic group with two elements, 
and the dihedral group with eight elements, respectively.
In proving Theorem~\ref{thm:GaloisMain} we will follow~\cite[Section 3.4]{Fes16}.

\begin{remark}\label{r:Rtsh}
Recall that $c_0, c_1, c_2$ are the roots of the polynomial $h$ defined in~\eqref{eq:h}.
Notice that $h$ has discriminant $\Delta=-16(D^3+27)=(4\zeta_4\beta_0\beta_1\beta_2)^2=\delta^2$.
In particular, $\Delta$ is nonzero, so all the roots are distinct.

It is possible to explicitly write the roots $c_1,c_2$ in terms of the  elements $\zeta_{12},\beta_0,\beta_1,\beta_2,$ and $c_0$.
Namely, the other roots of $h$ are
$$
    c_1 = \frac{-t-c_0 + \epsilon}{2} \textrm{ and } c_2 = \frac{-t-c_0 - \epsilon}{2},
$$ 
where 
$
\epsilon=\frac{\delta}{c_0(3c_0+2t)}.
$
\end{remark}

We now let $E$ denote the field $K(\delta, c_0) \subset L$, and let $F$ denote
the field $K(\beta_0)\subset L$, and finally let $J$ denote the field $K(\beta_1,\beta_2)\subset L$.

\begin{lemma}\label{l:3GalExt}
The following statements hold.
\begin{enumerate}
\item The extension $E/K$ is a Galois extension of degree $6$ with Galois group $\Gal (E/K)\cong S_3$.
\item The extension $F/K$ is a Galois extension of degree $2$ with
Galois group $\Gal (F/K) \cong \Z/2\Z$.
\item The extension $J/K$ is a Galois extension of degree $8$ with
Galois group $\Gal (J/K) \cong D_4$.
\item The fields $E,F,$ and $J$ intersect pairwise trivially; that is,
the intersection of any two of them equals $K$.
\item The compositum field $E\cdot F\cdot J$ equals $L$.
\end{enumerate}
\end{lemma}

\begin{figure}[!ht]
\xymatrix{ 
&  E\cdot F\cdot J = L&  &  &  
\\
 E=K(\delta, c_0) \ar@{-}[ur] &  &J=K(\beta_1, \beta_2) \ar@{-}[ul] &  
\\
   &      &  K( \beta_1)  \ar@{-}[u]&  K( \beta_2)   \ar@{-}[ul]
\\
   K(\delta) \ar@{-}[uu] &  F=K(\beta_0) \ar@{-}[uuu]   &   K(\zeta_3)  \ar@{-}[u] \ar@{-}[ur] &
\\
   &   K = \QQ (t)  \ar@{-}[ul]\ar@{-}[u] \ar@{-}[ur]  &     &  
} 
\caption{An alternative description of $L$.}
\label{fig:Laltdef}
\end{figure}
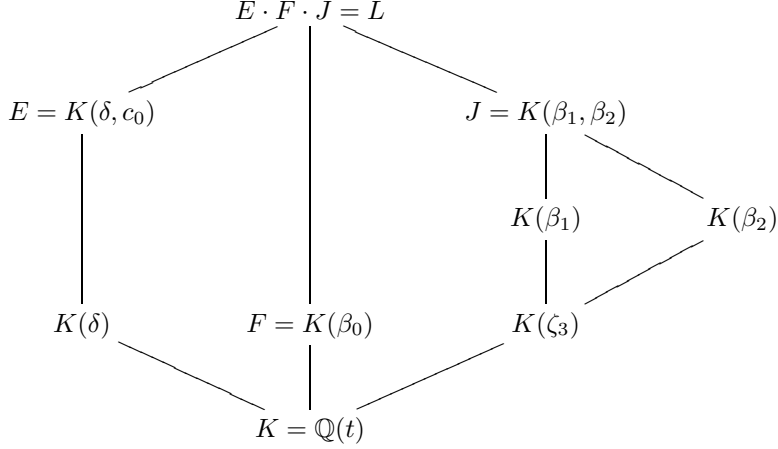

\begin{proof}
\begin{enumerate}[wide, labelwidth=!, labelindent=0pt]
\item By construction, the field $E$ is the splitting field of the cubic polynomial $h=x^3+t x^2+4$, that is irreducible over $K$ and whose discriminant is not a square in $K$.
The statement follows.
\item The field $F$ is the splitting field of the degree two polynomial $x^2-(3+t)$.
The statement trivially follows.
\item The field $J$ is the splitting field of the polynomial
$$
l=x^4 + (-2t + 3)x^2 + t^2 - 3t + 9,
$$
and so $J/K$ is a Galois extension.
The roots of $l$ are $\pm \beta_1, \pm \beta_2$, 
therefore the Galois group $\Gal (J/K)$ is generated by $\gamma_1, \gamma_2, \gamma$,
where $\gamma_1$ changes the sign of $\beta_1$,
$\gamma_2$ changes the sign of $\beta_2$,
and $\gamma$ switches $\beta_1$ and $\beta_2$.
Since $J/K$ is Galois, we have the following chain of equalities: $\#\Gal (L/K)=[L:K] = 8$.
One can easily check that $\gamma \gamma_1 \neq \gamma_1 \gamma$,
and that these two are the only elements of order $4$ of $\Gal (J/K)$.
Summarising, $\Gal (J/K)$ is a non-abelian group of order $8$ 
with exactly two elements of order $4$.
Thus $\Gal (J/K)$ is isomorphic to $D_4$. 
\item By explicit computations using MAGMA (\cite{magma}), code available at \cite{BCFNW17}.

As an example, we show the argument to prove the statement for $F$ and $J$.
By point \textit{(ii)}, the field $F\cap J$ is an extension of degree at most $2$ over $K$.
If we assume that $F\cap J \neq K$, it follows that $F\cap J = F$ and, therefore, that $\beta_0 \in J$.
By writing $J$ as 
$$\QQ \oplus \zeta_3\QQ \oplus 
\beta_1\QQ \oplus \zeta_3\beta_1\QQ \oplus
\beta_2\QQ \oplus \zeta_3\beta_2\QQ \oplus
\beta_1\beta_2\QQ \oplus \zeta_3\beta_1\beta_2\QQ$$
one can show, after some computations, that the equation $t^2+3$ has no solution in $J$, i.e., $\beta_0\notin J$, getting a contradiction.
We can hence conclude that $F\cap J=K$.
\item Recall that $L$ is, by definition, the Galois closure of the smallest field extension of $K$ containing $\zeta_{12}, \beta_i, c_i$, for $i=0,1,2$.
Hence notice that also  $\delta$ is in $L$.

Let $L'$ denote the compositum field $E\cdot F \cdot J$.
By construction, $L'$ contains the elements, $\delta, c_0, \beta_0, \beta_1, \beta_2$.
Then, by Remark~\ref{r:Rtsh}, we have that also $c_1$ and $c_2$ are in $L'$.
From $\beta_i\in L'$, it follows that $\frac{1}{3}(\beta_1^2-t)=\zeta_3\in L$.
Hence,  the ratio
$$
\alpha :=\frac{4\zeta_3 \beta_0\beta_1\beta_2}{4\zeta_4 \beta_0\beta_1\beta_2} = \frac{\zeta_3}{\zeta_4}
$$
is also inside $L'$.
Recall that $\zeta_3=\zeta_{12}^4$ and $\zeta_4=\zeta_{12}^3$,
then 
$$
\alpha = \frac{\zeta_3}{\zeta_4} = \frac{\zeta_{12}^4}{\zeta_{12}^3}= \zeta_{12}.
$$
Hence $\zeta_{12}$ is an element of $L'$.
It follows that $L'$ is the smallest field extension of $K$ containing $\zeta_{12}, \beta_i, c_i$, for $i=0,1,2$.

In order to show that $L=L'$ it suffices to show that the extension $L'/K$ is
        Galois.
This follows from the fact that $L'$ is the compositum of three Galois extensions of $K$ that intersect pairwise trivially.
Hence $L'$ is Galois over $K$ and therefore it equals its Galois closure,
i.e., $L=L'$.
\end{enumerate}
\end{proof}

\begin{proof}[Proof of Theorem~\ref{thm:GaloisMain}]
The fact that $L$ equals the field $K(\zeta_{12}, \beta_0, \beta_1, \beta_2, c_0)$ immediately follows from Lemma~\ref{l:3GalExt}.(5).

The field extension $L/K$ is Galois by definition. 
Being Galois, its degree equals the cardinality of its Galois group.

From Lemma~\ref{l:3GalExt}.(5), $L$ equals the compositum of the fields $E,F,$ and $J$;
from Lemma~\ref{l:3GalExt}.(4), these fields are pairwise distinct,
and so 
$$
\Gal (L/K) \cong \Gal (E/K) \times \Gal (F/K) \times \Gal (J/K).
$$
    The theorem now follows from Lemma~\ref{l:3GalExt} (1)-(3).
\end{proof}

\begin{remark}
In order to perform explicit operations using the automorphisms of $\Gal (L/K)$, it is useful to give an isomorphism
between $\Gal (L/K)$ and $S_3 \times \Z/2\Z \times D_4$.
In order to do so, 
we will present five automorphisms $\tau_i\in \Gal (K_2/K)$, with $i=1,2,3,4,5$,
such that:
\begin{align*}
\Gal (E/K) = \langle \tau_1, \tau_2 \rangle &\cong S_3; \\
\Gal (F/K) = \langle \tau_3 \rangle &\cong \Z/2\Z;\\
 \Gal (J/K) = \langle \tau_4, \tau_5 \rangle &\cong D_4. 
\end{align*}

The field $L$ is generated by $c_0, \zeta_{12}, \beta_0, \beta_1, \beta_2$ over $K$,
so to describe an element $\tau\in \Gal (L/K)$ 
it is enough to describe  its action on those elements.
The action of $\tau_i$ on those generators of $L$ over $K$ is listed in the table below.
For the convenience of the reader, 
the table also lists the action of $\tau_i$, for $i=1,...,5$, on other interesting elements of $L$.

\vspace*{5mm}
\begin{table}[!h]
\begin{tabular}{| c || c | c | c | c | c | c | c | c | c | c |}
 \hline
 & $c_0$ & $c_1$ & $c_2$ & $\delta$  & $\zeta_{12}$  & $\zeta_4$ & $\zeta_3$ & $\beta_0$ & $\beta_1$ & $\beta_2$  \\
 \hline
 \hline
 $\tau_1$ & $c_0$ & $c_2$ & $c_1$ & $-\delta$  & $\zeta_{12}^7$  & $-\zeta_4$ & $\zeta_3$ & $\beta_0$ & $\beta_1$ & $\beta_2$  \\
 \hline
 $\tau_2$ & $c_1$ & $c_2$ & $c_0$ & $\delta$  & $\zeta_{12}$  & $\zeta_4$ & $\zeta_3$ & $\beta_0$ & $\beta_1$ & $\beta_2$  \\
 \hline
 $\tau_3$ & $c_0$ & $c_1$ & $c_2$ & $\delta$  & $\zeta_{12}^7$  & $-\zeta_4$ & $\zeta_3$ & $-\beta_0$ & $\beta_1$ & $\beta_2$  \\
 \hline
$\tau_4$  & $c_0$ & $c_1$ & $c_2$ & $\delta$  & $\zeta_{12}^{11}$  & $-\zeta_4$ & $\zeta_3^2$ & $\beta_0$ & $-\beta_2$ & $\beta_1$  \\
 \hline
 $\tau_5$ & $c_0$ & $c_1$ & $c_2$ & $\delta$  & $\zeta_{12}^7$  & $-\zeta_4$ & $\zeta_3$ & $\beta_0$ & $\beta_1$ & $-\beta_2$  \\
 \hline
\end{tabular}
\vspace{2mm}
\caption{The action of $\tau_1,...,\tau_5$ on a set of elements of $L$.}
\label{table:GalAut}
\end{table}

\end{remark}

\vskip.5in

\bibliographystyle{alpha}
\bibliography{biblio}

\end{document}